\documentclass[11pt,a4paper]{article}
\usepackage{amsmath,amsthm,amssymb,amsfonts}
\usepackage{hyperref}
\usepackage{enumitem}

\newcommand{\C}{{\mathbb{C}}}          
\newcommand{\Proj}{{\mathbb{P}}}        
\newcommand{\R}{{\mathbb{R}}}          
\newcommand{\Z}{{\mathbb{Z}}}          
\newcommand{\Sphere}{\mathbb{S}}
\newcommand{\Hyperbolic}{\mathbb{H}}

\newcommand{\fraku}{{\mathfrak{u}}}
\newcommand{\frakv}{{\mathfrak{v}}}

\newcommand{\XIS}{{\mathfrak{X}}}

\newcommand{\SO}{{\mathrm{SO}}}

\newcommand{\SU}{{\mathrm{SU}}}

\newcommand{\lrr}{\longrightarrow}

\newcommand{\na}{{\nabla}}
\newcommand{\End}[1]{{\mathrm{End}}\,{#1}}

\newcommand{\deter}[1]{{\mathrm{det}}\,{#1}}

\newcommand{\dx}{{\mathrm{d}}}

\newcommand{\papa}[2]{\frac{\partial#1}{\partial#2}}

\newcommand{\vol}{{\mathrm{vol}}}
\newcommand{\ric}{{\mathrm{Ric}\,}}

\newcommand{\estrela}{{\boldsymbol{\star}}}

\newtheorem{teo}{Theorem}
\newtheorem{lemma}{Lemma}
\newtheorem{coro}{Corollary}
\newtheorem{prop}{Proposition}

\newenvironment{Rema}[1][Remark.]{\begin{trivlist}
\item[\hskip \labelsep {\bfseries #1}]}{\end{trivlist}}

\def\cyclic{\mathop{\kern0.9ex{{+}
\kern-2.2ex\raise-.28ex\hbox{\Large\hbox{$\circlearrowright$}}}}\limits}

\title{On the volume of a unit vector field in 3 dimensions via calibrations}

\author{Rui Albuquerque}

\begin{document}


\maketitle


\begin{abstract}

We give a new proof of the well-known result that the minimal volume vector fields on $\Sphere^3(r)$ are the Hopf vector fields. Such proof relies again on calibration theory, arising here from a systematic point of view given by a natural source of differential forms. Our results serve in particular for all $r$.

A classification of relevant calibrations on $T^1M$ for every oriented 3-manifold $M$ of constant sectional curvature is given, continuing the study of the \textit{usual} fundamental differential system of Riemannian geometry. Showing applications of this differential system is also one of the purposes of this article.

We deduce new properties of the geodesic flow vector field of space forms, which interacts with the solutions of the minimal volume problem both in elliptic and hyperbolic geometry, in any dimension. The solution -- unknown -- for the hyperbolic case in 3-dimensions being most dependent on the homology class of the domain and boundary values of the vector fields. This is illustrated with a noteworthy example which ironically works just for curvature $-1$.

\end{abstract}


\
\vspace*{5mm}
\vspace*{2mm}

\vspace*{6mm}

\markright{\sl\hfill  Rui Albuquerque \hfill}

\setcounter{section}{1}

\begin{center}
\subsection*{{1 -- Introduction}}
\end{center}

\begin{center}
 \textsc{1.1 -- A well-known result}
\end{center}

In the article \cite{GluckZiller}, H.~Gluck and W.~Ziller applied the theory of calibrations to determine the unit vector fields of minimal volume defined on the unit sphere $\Sphere^3$.

Given a Riemannian manifold $(M,g)$ of dimension $m$, the volume of a unit-norm vector field $X\in\XIS_M$ is defined
as $\vol(X)=\vol(M,X^*g^S)$. This is, the volume of the embedded submanifold $X(M)$ of the unit
tangent sphere bundle of $M$ endowed with the Sasaki metric $g^S$, and as long as the real quantity
exists. It is well-known that
\begin{equation}  \label{Definition_volume}
 \begin{split}
&\vol(X)=\int_M\sqrt{\deter(1+(\na X)^t\na X)}\,\vol =\int_M\biggl(1+\sum_{j=1}^m\|\na_{e_j}X\|^2 \biggr. \\
&\qquad\quad \biggl.+\sum_{j_1<j_2}\|\na_{e_{j_1}}X\wedge\na_{e_{j_2}}X\|^2+\cdots
+\sum_{j_1<\cdots<j_{m-1}}\|\na_{e_{j_1}}X\wedge\cdots\wedge\na_{e_{j_{m-1}}}X\|^2\biggr)^{\frac{1}{2}}\,\vol
\end{split}
\end{equation}
where $\{e_1,\ldots,e_m\}$ is an orthonormal local frame on $M$, cf. \cite{BritoChaconNaveira, GilMedranoLlinaresFuster,GluckZiller}.

The celebrated article describes the construction of a calibration 3-form $\varphi$ on the total
space of the unit tangent sphere bundle $\pi:T^1\Sphere^3\lrr \Sphere^3$, over the unit radius
3-sphere. Sections of the bundle correspond to embedded submanifolds and it is observed that
non-vanishing vector fields define a unique homology class. Any Hopf vector field on the 3-sphere
turns out to be a calibrated submanifold for $\varphi$ and therefore a minimal volume vector field
in its homology class (\cite{HarLaw}).

Let us describe the essential steps taken in \cite{GluckZiller} for the construction of the 3-form
$\varphi$. First, the authors recall the circle bundle
\begin{equation}
\begin{split}
 & \mathrm{V}_2(\R^4)=T^1\Sphere^3\lrr\mathrm{Gr}_2(\R^4) \\
 &\qquad\qquad (x,y)\longmapsto x\wedge y
\end{split}
\end{equation}
mapping the Stiefel manifold of pairs of orthogonal units $(x,y)$ to the Grassmannian of oriented
planes. The latter is the quadric $z_1^2+\cdots+z_4^2=0$ in $\C\Proj^3$, a K\"ahler surface embedded
as $z_i=x_i+\sqrt{-1}y_i$. The vertical direction of the bundle induces a 1-form $\theta$, dual to a
horizontal direction $e_0$ of the fibration $T^1\Sphere^3\lrr\Sphere^3$. We shall see next that this
corresponds with the velocity of the geodesic flow and thus integrates precisely to the Hopf
fibration.

We see in the original paper that $\theta$ is a factor of the 3-form $\varphi$. It is also recalled
that $\mathrm{Gr}_2(\R^4)=\Sphere^2\times \Sphere^2$ isometrically, with both spheres of radius
$\sqrt{2}$.

Finally, taking the pullback of the two K\"ahler forms from $\Sphere^2$, one finds two
linearly independent 2-forms $\tau_1$ and $\tau_2$, necessarily closed forms. These are essential to the definition of $\varphi$:
\[  \varphi=\theta\wedge(\tau_1-\tau_2) . \]
The contact 1-form $\theta$ and the two orthogonal closed 2-forms on $T^1\Sphere^3$
are invariant under the isotropy subgroup $\SO(2)\times\SO(2)$ of the action of
$\SO(4)\times\SO(2)$ on the 5-dimensional space.

However, we discover that the decomposition of the 4-dimensional $e_0^\perp=\ker\theta$ into two
2-planes is not the decomposition into `horizontal plus vertical' of the Sasaki metric, induced by
the Levi-Civita connection of $\Sphere^3$. It so happens that this is the source of confusion
with the involved natural exterior differential systems on $T^1\Sphere^3$ and thus it is our aim to clarify the problem.

\begin{center}
 \textsc{1.2 -- New motivation}
\end{center}

The above solution seems to be and in fact is very particular to the 3-sphere. First, already it is proved in \cite{GluckZiller} that the method does not work for $\Sphere^5$ and above, 
even though it is observed  that a calibration on any $T^1\Sphere^{2n+1}$ could also be defined,
this time invariant under $\SO(2n+2)\times\SO(2)$. Secondly, besides the presence of a rank 1
symmetric space, implying 2-point homogeneous action, it is known the geodesic spray of a given
Riemannian manifold acts by isometries (on the unit tangent bundle) if and only if the manifold has constant sectional curvature 1, \cite[p.345]{Chavel}. As we shall prove again below, the geodesic spray is directly related to the Hopf vector fields.

Furthermore, sectional curvature 1 is equivalent to the general Riemannian fibration induced by the geodesic spray and Sasaki metric being of K\"ahler base, a result due to Y.~Tashiro, cf. \cite{Alb2019b,Alb2020a}. In other words, equivalent to the canonical tangent almost contact structure on $T^1\Sphere^m$ verifying the conditions of a Sasakian manifold.

The relation between radius $r$ on the base and radius 1 on the fibre is only considered in \cite{BorGil2006}. Certain stability results are deduced depending on the curvature of the base, although again just with $\dim\geq 5$.

As it is well known, the action of $\SO(4)$ on $\Sphere^3$ by isometries is also transitive and
isometric when lifted to the symmetric space $(T^1\Sphere^3,g^S)$ as the differential map. For
$(x,y)\in V_2(\R^4)$ and $g\in\SO(4)$, we have $g(x,y)=(gx,gy)$, which clearly commutes with
the geodesic flow, the $\SO(2)$ action of the great circles, as referred in \cite{GluckZiller}:
\begin{equation}\label{geodesic_flow_on_Sm}
 g_t(x,y)=(x\cos t+y\sin t,-x\sin t+y\cos t)  .
\end{equation}

This last result is true in any dimension and we prove it later, in Theorem \ref{Theorem_wideraction}, in the wider context of constant sectional curvature $c\neq0$. Indeed there exists a similar description of the geodesic flow for the hyperbolic metric, which \textit{is not} found in the literature:
\begin{equation}
  g_t(x,y)=(x\cosh t+y\sinh t,x\sinh t+y\cosh t) .
\end{equation}
In Theorem \ref{Theorem_velocity_proof} we prove that, both on ${T^1(\Sphere^m(r))}$ and on $T^1(\Hyperbolic^m(r))$, the velocities of those curves $g_t(x,y)$ agree with a multiple by $r$ of the geodesic spray vector field $e_0$.

Finally, since the theory of calibrations is also valid with a fixed boundary volume, that should be relevant for the question of minimality of vector fields on a hyperbolic domain of finite volume.

We start by describing all invariant calibrations on $T^1M$ for 3-di\-men\-sional space forms $M$. This contributes to our study of minimal volume vector fields through calibrated geometry, and its developments, as well as to the knowledge of the so-called fundamental differential system of global $n$-forms on the tangent sphere bundle of any given oriented Riemannian manifold of dimension $(n+1)$, cf. \cite{Alb2018a,Alb2019b}. Within the new applications of such differential system we have the study of unit vector fields.

The contact 1-form $\theta$ referred above is precisely the same for \textit{our} differential system, which is moreover invariant, also, by any lift of any isometry of the base. As it was expected, the above picture within K\"ahler geometry is a particular case of our calibration systems. The assertion that ``there is, up to constant multiple, just one isometry-invariant closed 3-form'', found in \cite[p.178]{GluckZiller}, seems to be justified by the acting subgroup in question.

To great extent we have a new proof of Gluck and Ziller's celebrated result, following Theorem \ref{Theorem_Hopfminimalforanycurvature} below. The general result holds for \textit{most} positive constant sectional curvature manifolds, ie. for all 3-dimensional quotients of $\Sphere^3(r)$ which admit a Hopf vector field. It has the first proof given entirely with calibrated geometry.

\begin{center}

\subsection*{2 -- A different exterior differential system}

\end{center}

\begin{center}
\textsc{2.1 -- The fundamental differential system of Riemannian geometry in dimension 3}
\end{center}

The subtitle refers to the fundamental exterior differential system of Riemannian geometry introduced in \cite{Alb2018a,Alb2019b}, here in the case of dimension 3.

Let $(M,g)$ be an oriented Riemannian 3-manifold. We have the well-known Sasaki metric and the natural contact structure $\theta=e^0$ on the total space of $\pi:T^1M\lrr M$. In the nomenclature of \cite{KatokHasel}, we say $T^1M=\{u\in TM:\ \|u\|^2=1\}$ is a submanifold of contact type of the manifold $TM$ with its canonical symplectic structure induced from $T^*M$ by metric duality.

One proves easily that $TT^1M=\xi^\perp_{|T^1M}$, where $\xi=Be_0$ is the tautological vertical vector field over $TM$. Then we have the splitting of $TTM$ into `horizontal plus vertical' subbundles. The mirror map $B\in\End{TTM}$ is canonical and confirms the structure group of $(T^1M,g^S)$ indeed reduces from $\SO(5)$ to $\SO(2)$.

Let $e_0,e_1,e_2,e_3,e_4,\ e_3=Be_1,\ e_4=Be_2$, be an adapted frame of $TT^1M$, meaning $e_0$ is the geodesic spray, horizontal --- indeed the unique horizontal vector field such that $\dx\pi_u(e_0)=u\in T_{\pi(u)}M$, for all $u\in TM$, --- and meaning $e_1,e_2$ span the horizontal plane bundle $H^\na\cap e_0^\perp$ and their mirror $e_3,e_4$ span the vertical plane bundle, tangent to the fibres $\Sphere^2$ of $\pi$.

The orientation on $T^1M$ is given by $\vol_{T^1M}=e^{01234}$, independent of that one on $M$. Finally, the fundamental differential system is composed of the following 2-forms globally defined on $T^1M$:
\begin{equation} \label{thealphaiwithadaptedframe}
 \dx\theta=e^{31}+e^{42},\qquad\alpha_0=e^{12}, \qquad\alpha_1=e^{14}-e^{23},\qquad\alpha_2=e^{34} .
\end{equation}
Notice we require the orientation of $M$. It follows that, $\forall i=0,1,2$,
\begin{equation}
\begin{split}
 & \quad\ \alpha_i\wedge\dx\theta=\alpha_0\wedge\alpha_1=\alpha_2\wedge\alpha_1=0,\qquad\alpha_1\wedge\alpha_1=-2\alpha_0\wedge\alpha_2=(\dx\theta)^2 , \\
 &  *\alpha_0=\theta\wedge\alpha_2,\qquad *\alpha_1=-\theta\wedge\alpha_1,\qquad *\alpha_2=\theta\wedge\alpha_0,  \qquad *\dx\theta=-\theta\wedge\dx\theta .
\end{split}
\end{equation}
Such are the basic structural equations of the differential system associated to any given oriented Riemannian 3-manifold.

We assume now that $M=M_c$ has constant sectional curvature $c\in\R$. The derivatives
\begin{equation}
 \dx\alpha_0=\theta\wedge\alpha_1,\qquad\dx\alpha_1=2\theta\wedge\alpha_2-2c\,\theta\wedge\alpha_0,\qquad\dx\alpha_2=-c\,\theta\wedge\alpha_1
\end{equation}
have become well-known.

\begin{center}
\textsc{2.2 -- Invariant calibrations}
\end{center}

By straightforward computation or by \cite[Lemma 2.1]{Alb2020a}, a 2-form arising as a constant coefficient linear combination of the invariant differential system \eqref{thealphaiwithadaptedframe} on $M=M_c$ is closed if and only if there exist $Q,Q_1\in\R$ such that
\begin{equation} \label{omega_c}
 \omega_c:=cQ\alpha_0+Q\alpha_2+Q_1\dx\theta .
\end{equation}

We are now interested in a calibration 3-form $\varphi$ on $T^1M$ which is a multiple of the geodesic flow (we only have $\theta$ and the invariant 2-forms to work with), this is, the wedge of $\theta$ with a 2-form $\omega$ of comass 1 on $e_0^\perp$. The definition of comass 1, for any form, is that $\varphi\leq\vol$ and that equality is attained as a supremum on simple 3-vectors $\fraku\wedge\frakv\wedge\mathfrak{w}$, formed of unit tangent directions at the points of $T^1M$, cf. \cite[p.57]{HarLaw}.

Let $\omega=b_0\alpha_0+b_1\alpha_1+b_2\alpha_2+b_3\dx\theta$, with $b_0,\ldots,b_3\in\R$ constant.

By straightforward computation or by checking \cite[Lemma 2.2]{Alb2020a}, we find that $\varphi=\theta\wedge\omega$ is closed if and only if it is of the kind
\begin{equation}
\varphi= \theta\wedge(b_0\alpha_0+b_1\alpha_1+b_2\alpha_2) .
\end{equation}
We remark that there are closed \textit{invariant} 3-forms in the case of the base space with any sectional curvature; simply, they must also satisfy $b_2=0$.

The requirement that $\omega$ is a symplectic form is not mandatory in order to have a calibration, but seems to lead to the less restrictive conditions on the $\varphi$-calibrated submanifolds.

Let us first suppose the orientations on $e_0^\perp$ induced by $(\dx\theta)^2=\dx\theta\wedge\dx\theta=-2e^{1234}$ and $\omega\wedge\omega$ give the same volume. We find that $\omega$ must satisfy
\begin{equation}
 b_0b_2-{b_1}^2-{b_3}^2=-1.
\end{equation}
\begin{prop} \label{circleofcalibrations}
(i)\  \,$\omega\wedge\omega=-2e^{1234}$ if and only if $b_0b_2-{b_1}^2
 -{b_3}^2=-1$\\
(ii) an invariant 3-form $\varphi=\theta\wedge\omega$ as above defines a calibration on $T^1M$ if and only if
 \begin{equation}   \label{invariantthreeformcalibration}
     {b_0}^2+{b_1}^2=1,\qquad b_2=-b_0,\qquad b_3=0 .
 \end{equation}
\end{prop}
\begin{proof}
 As mentioned, $\dx\varphi=0$ if and only if $b_3=0$. Now we must have ${b_1}^2-b_0b_2=1$. The calibration $\varphi$ has norm at most $\sqrt{2}$, the maximum being attained where $\omega$ is symplectic (due to Wirtinger's inequality). This is equivalent to $\omega\wedge*\omega\leq2\vol_{T^1M}$. From
 \begin{equation*}
 \begin{split}
 \langle\omega,\omega\rangle\vol_{T^1M} &= (b_0\alpha_0+b_1\alpha_1+b_2\alpha_2)\wedge\theta\wedge (b_2\alpha_0-b_1\alpha_1+b_0\alpha_2) \\
 & = ({b_0}^2+{b_2}^2+2{b_1}^2)\vol_{T^1M}
   \end{split}
 \end{equation*}
 we find ${b_0}^2+{b_2}^2+2{b_1}^2\leq2$. Hence $(b_0+b_2)^2\leq0$ and therefore ${b_0}^2+{b_1}^2=1$. 
\end{proof}

Notice that $\varphi=\theta\wedge(b_0\alpha_0+b_1\alpha_1-b_0\alpha_2)=-*\omega$. The following result is more surprising. Recall $\omega_c$ from \eqref{omega_c}.
\begin{prop} \label{omegaclosedbadorientation}
 The 2-form $\omega$ is closed if and only if $c=-1$, $b_0=-b_2=\pm1$ and $b_1=0$.
\end{prop}
\begin{proof}
 $\omega=\omega_c\neq0$ if and only if $b_1=0$ and $cQ=b_0=-Q$; hence $b_0=\pm1$ and $c=-1$.
\end{proof}

Now keeping the orientation on $T^1M$ and choosing that on $e_0^\perp$ to be the opposite of $(\dx\theta)^2$, we have the following results analogous to the above.
\begin{prop} \label{omega_goodorientation}
(i)\  \,$\omega\wedge\omega=2e^{1234}$ if and only if $b_0b_2-{b_1}^2
 -{b_3}^2=1$\\
(ii)\ with such $\omega$, the 3-form $\varphi=\theta\wedge\omega$ is a calibration if and only if
$b_0=b_2=\pm1,\ b_1=b_3=0$\\
 (iii)\ moreover $\omega$ is closed if and only if $c=1$.
\end{prop}
\begin{proof}
 Notice that comass 1 implies all $|b_i|\leq1$. Now, in case (ii), we have $b_0b_2={b_1}^2+1$. The result follows.
\end{proof}

\begin{Rema}[Remarks.]
  1)\ As referred above, we do have two symplectic forms $\omega$ with either orientation on $e_0^\perp$. In case of \eqref{invariantthreeformcalibration} it is written in an orthonormal frame as $ \omega=e^1\wedge(b_0e^2+b_1e^4)+e^3\wedge(b_1e^2-b_0e^4)$. And in the case of the opposite orientation, $\pm\omega=e^{12}+e^{34}$. \,2)\
  In the last case, again, we have $\varphi=\theta\wedge\omega=-*\omega$ on the 5-dimensional space. \,3)\  Notice the uniqueness of invariant $\varphi$ is assured only when we ask that $\omega$ too is closed.  \,4)\ We stress that there are no calibrations of the kind given by Propositions \ref{omegaclosedbadorientation} or \ref{omega_goodorientation}(iii) on a flat manifold.
\end{Rema}

Notice that $\theta\wedge\dx\theta$ cannot contribute to a closed invariant 3-form. Therefore, checking the ``Table of invariant forms and their exterior derivatives'' from \cite[p.182]{GluckZiller}, which is specific to the 3-sphere case $c=1$, we conclude our differential system must sit perpendicularly to theirs, since we do not obtain the \textit{two} closed 2-forms on $T^1M$ which add to give a calibration. The subbundle $e_0^\perp$ is the same, the 2-forms there are skew.

\begin{center}
\textsc{2.3 -- Exact and cohomologous calibrations}
\end{center}

Here we find out when do two invariant calibrations on $T^1M$ define the same invariant cohomology class ($\sim$). By invariance of the structure one may restrict to a cohomology with constant coefficients. So we consider the closed 3-forms
\begin{equation}
\varphi_t=\theta\wedge(b_{0t}\alpha_0+b_{1t}\alpha_1+b_{2t}\alpha_2)\quad\mbox{and}\quad
\varphi=\theta\wedge(b_{0}\alpha_0+b_{1}\alpha_1+b_{2}\alpha_2).
\end{equation}

\begin{prop} \label{cohomologous3forms}
 \,\ $\varphi_t\sim\varphi\ \Longleftrightarrow\ b_{0t}-b_{0}=-c(b_{2t}-b_2)$.
\end{prop}
\begin{proof}
We have $\varphi_t=\varphi+\dx(\sum_ip_i\alpha_i)$ for some constant $p_0,p_1,p_2$ if and only if
\begin{equation*}
 \varphi_t-\varphi
   = \theta\wedge(p_0\alpha_1+2p_1\alpha_2-2cp_1\alpha_0-p_2c\alpha_1) .
\end{equation*}
This is equivalent to the system
$b_{0t}-b_{0}=-2cp_1,\ \, b_{1t}-b_{1}=p_0-2p_2c,\ \, b_{2t}-b_{2}=2p_1$, whose solutions exist under the referred condition.
\end{proof}

We restrict to the two types of calibration found before:
\begin{equation}
 \varphi_t=\theta\wedge(\cos{t}\,\alpha_0+\sin t\,\alpha_1-\cos t\,\alpha_2) \qquad\mbox{and}\qquad \varphi_+=\theta\wedge(\alpha_0+\alpha_2) .
\end{equation}
Notice that $\varphi_+=\theta\wedge\omega_{1}$ and $\varphi_-=\varphi_0=\theta\wedge\omega_{-1}$, recalling $\omega_c$ from \eqref{omega_c} with $Q=\pm1,Q_1=0$.

The theory of exact calibrations is relevant on its own, hence the following results. By Proposition \ref{cohomologous3forms}, we have $\varphi_t\sim \varphi_0=\theta\wedge(\alpha_0-\alpha_2)$ iff $\cos t-1=c(\cos t-1)$. And $\varphi_t\sim\varphi_+$ iff $\cos t-1=c(\cos t+1)$.

\begin{prop} \label{cohomologous3formsc=1}
If $c=1$, then all $\varphi_t\sim\varphi_0\sim 0$ and  $\varphi_+\nsim 0$.
\end{prop}

\begin{prop}  \label{cohomologous3formscneq1}
If $c\neq1$, then:\\
(i)\ \,$\varphi_t\sim\varphi_{s}$ iff $\cos t=\cos s$\\
(ii)\,\ $\varphi_t\sim0$ iff  $t={\frac{\pi}{2}}\mod\Z\pi$\\
(iii)\,\ $\varphi_t\sim\varphi_+$ iff $\cos t=\frac{1+c}{1-c}=-1+\frac{2}{1-c}$; in particular, $\varphi_t\nsim\varphi_+$ for $c>0$\\
(iv)\ \,$\varphi_+\sim0$ iff $c=-1$.
\end{prop}

Once again we observe the particular cases of $c=\pm1$. If $c=0$, then $\varphi_-=\varphi_0\sim\varphi_+\nsim0$.

\begin{center}

\subsection*{3 -- Unit vector field and calibrations}

\end{center}

\begin{center}
\textsc{3.1 -- Solutions with minimal volume}
\end{center}

Inspired by \cite{GluckZiller}, we apply the theory of calibrations to the search for the minimal volume vector fields on oriented constant sectional curvature 3-manifolds.

Let $X\in\XIS_M$ be a unit norm vector field on a given Riemannian $(n+1)$-manifold. For the moment it is easy to work in general dimension.

Within the homology class of $X(M)\subset T^1M$, the absolute minimal volume unit vector fields are those for which $\varphi=\vol$ when restricted to the submanifold. In other words, $X$ has minimal volume when $X(M)$ is a calibrated or $\varphi$-submanifold of the pair $(T^1M,\varphi)$. Recalling $\vol_X$ from \eqref{Definition_volume}, the equation is equivalent to
\begin{equation}
X^*\varphi=\vol_X.
\end{equation}

Certainly, every vector field induces a monomorphism $H_{n+1}(M)\hookrightarrow H_{n+1}(T^1M)$. And indeed we have a well-known fundamental relation, cf. \cite{HarLaw}: for any unit $X'\in\XIS_M$ in the same homology class as $X$,
\begin{equation}
   \int_M\vol_X =\int_{X(M)}\varphi=\int_{X'(M)}\varphi \leq
   \int_M\vol_{X'}.
\end{equation}

The theory of calibrations holds for submanifolds with boundary. One may choose a fixed open subset, a domain $\Omega\subset M$ perhaps with non-empty boundary, and seek for an immersion $X:\Omega\lrr T^1M$ giving a $\varphi$-submanifold. Prescribing values for $X$ on a given compact boundary $\partial\Omega$ implies that certain \textit{moment} conditions are satisfied, cf. \cite[(6.9)]{HarLaw}.

This observation also tells that we may take just the invariant $\varphi$ when considering a homogeneous Riemannian manifold. Since the minimal solutions will be invariant too.

Now, following  e.g. notation from \cite{Alb2018a,Alb2019b}, we use $\pi^*X,\pi^\estrela X$ (different stars) for the canonical lifts of a vector field $X$ on $M$. We write the well-known `horizontal plus vertical' decomposition in $TTM$
\begin{equation}
 \dx X(Y)=\pi^*Y+\pi^\estrela(\na_YX).
\end{equation}

Since $T^1M=\{u\in TM:\ \|u\|^2=1\}$ is also characterized as the locus of $\|\xi\|^2=1$, where $\xi$ is the tautological vertical tangent vector field over $TM$, this is, $\xi_u=u$, and since $(\pi^\estrela\na)_Y\xi=Y^v$, it follows that $TT^1M=\xi^\perp$ for the Sasaki metric. Moreover, if $X$ is a unit norm vector field on $M$, then $\dx X(Y)\in TT^1M$.

Now let us recall the adapted oriented orthonormal frames $\{e_0,e_1,\ldots,e_n,e_{n+1},\ldots,e_{2n}\}$ on $T^1M$ varying with $u\in T^1M$. The geodesic spray ${e_0}$ equals $\pi^*X$ over the submanifold $X(\Omega)$, for any $X$. In other words, $X^*\theta=X^\flat$. The $e_i,\ i\leq n$, are the horizontal vectors and thus project at each $u\in T^1M$ to give a frame $e_i\in TM$. Hence we may write
\begin{equation}
\dx X(e_i)=e_i+\sum_{j=1}^n A_{ij}e_{j+n} ,
\end{equation}
for all $i=0,1\ldots,n$, where $A_{ij}=\langle\na_{e_i}X,e_j\rangle$.

We now resume with the study in dimension three.

It is not at all clear how the choice of a 3-form $\varphi=\theta\wedge\omega$ as in Section 2 may influence the existence of a solution vector field as calibrated submanifold, if it ever does.

\begin{lemma} \label{Lemma_varphicalibratedvectorfield}
 Let $\varphi=\theta\wedge(b_0\alpha_0+b_1\alpha_1+b_2\alpha_2)$. Then $X^*\varphi=\vol_X$ if and only if
 \begin{equation}  \label{varphicalibratedvectorfield}
 b_0+b_1(A_{11}+A_{22})+b_2(A_{11}A_{22}-A_{12}A_{21})=\sqrt{1+\sum_{i,j=0}^2{A_{ij}}^2+\sum_{j=0}^2{A_{(j0)}}^2}
 \end{equation}
 where the $A_{(j0)}$ are defined below.
\end{lemma}
\begin{proof}
  First, by definition of Riemannian immersion, we have the formula from \eqref{Definition_volume} as well as the helpful identity developed in \cite{BritoChaconNaveira} for all dimensions: $\vol_X=\|X_*(e_0\wedge e_1\wedge\cdots\wedge e_n)\|\,\vol_M$. Thus we compute
 \begin{eqnarray*}
  \lefteqn{X_*(e_0\wedge e_1\wedge e_2) =}  \\
  & &= (e_0+A_{01}e_3+A_{02}e_4)\wedge(e_1+A_{11}e_3+A_{12}e_4)\wedge(e_2+A_{21}e_3+A_{22}e_4) \\
  & & = e_{012}+A_{21}e_{013}+A_{11}e_{032}+A_{22}e_{014}+A_{12}e_{042} +A_{01}e_{312}+A_{02}e_{412}+\\
  & & \ \ \ \ \underbrace{(A_{11}A_{22}-A_{21}A_{12})}_{A_{(00)}}e_{034}+\underbrace{(A_{02}A_{21}-A_{01}A_{22})}_{-A_{(10)}}e_{134}+\underbrace{(A_{01}A_{12}-A_{02}A_{11})}_{A_{(20)}}e_{234}
 \end{eqnarray*}
 where $A_{(ij)}$ denote the minors of $(A_{ij})_{i,j=0,1,2}$. Then, since all $A_{i0}=0$,
 \begin{equation*}
  \|X_*(e_0\wedge e_1\wedge e_2)\|^2=1+\sum_{i,j=0}^2{A_{ij}}^2+\sum_{j=0}^2{A_{(j0)}}^2.
 \end{equation*}

 Considering $\varphi=\theta\wedge(b_0\alpha_0+b_1\alpha_1+b_2\alpha_2)$ and recalling \eqref{thealphaiwithadaptedframe} it is easy to deduce that $X^*\varphi(e_0,e_1,e_2)=b_0+b_1(A_{11}+A_{22})+b_2(A_{11}A_{22}-A_{12}A_{21})$.

 One may also compute $X^*\varphi$ directly from $X^*\theta=X^\flat=e^0$ and, moreover, $X^*e^i=e^i$, for $0\leq i\leq2$, and $X^*e^3=A_{01}e^0+A_{11}e^1+A_{21}e^2$ and $X^*e^4=A_{02}e^0+A_{12}e^1+A_{22}e^2$.
\end{proof}


We let
\begin{equation}\label{veryclosedcalibration}
 \varphi_\pm=\theta\wedge(\alpha_0\pm\alpha_2).
\end{equation}

\begin{prop} \label{Propfundamental}
 Let the oriented Riemannian manifold $(M,g)$ have constant sectional curvature (any constant). Then a vector field $X$ on the manifold $M$ corresponds to a $\varphi_\pm$-submanifold if and only if (i) $\na_XX=0$ and (ii)  $\langle\na_YX,Y\rangle=\pm\langle\na_ZX,Z\rangle$ and $\langle\na_YX,Z\rangle=\mp\langle\na_ZX,Y\rangle$, $\forall Y,Z$ such that $X,Y,Z$ is any orthonormal frame on $M$.
\end{prop}
\begin{proof}
Applying the lemma, if $\|X^*\varphi_\pm\|^2=(X^*\varphi_\pm(e_0,e_1,e_2))^2$ equals $\|X_*(e_0\wedge e_1\wedge e_2)\|^2$, then
 \begin{equation}  \label{normofvarphivectorfield}
      1\pm2A_{(00)}+{A_{(00)}}^2 = 1+\sum_{i,j=0}^2{A_{ij}}^2+\sum_{j=0}^2{A_{(j0)}}^2,
\end{equation}
this is,
\begin{equation*}
 {A_{01}}^2+{A_{02}}^2+{A_{11}}^2+{A_{12}}^2+{A_{21}}^2+{A_{22}}^2+ {A_{(10)}}^2+{A_{(20)}}^2\mp2A_{(00)}=0.
\end{equation*}
 Since $A_{(00)}=A_{11}A_{22}-A_{21}A_{12}$, we find
 \begin{equation*}
  {A_{01}}^2+{A_{02}}^2+(A_{11}\mp A_{22})^2+(A_{12}\pm A_{21})^2+
  {A_{(10)}}^2+{A_{(20)}}^2=0
 \end{equation*}
 giving immediately the three identities in (i-ii). The reciprocal, the three identities implying \eqref{varphicalibratedvectorfield}, is easily checked.
 \end{proof}

 Recall the notions of closed vector field, $\dx X^\flat=0$, of divergence-free or co-closed vector field, $\mathrm{tr}_g{\na X}=0$, and that of harmonic vector field (as a 1-form, not as a map).
 \begin{prop}
 Let $X$ be a solution of equations (i-ii) in Proposition \ref{Propfundamental}. We have that:
\begin{itemize}
 \item in the case $\varphi_+$, then $X$ is a Killing vector field and therefore co-closed. It is not closed and thus not a harmonic vector field.
 \item (cf. Proposition \ref{prop_nonexistenceminimalvolumevectorfield} below) in the case $\varphi_-$, then $X$ is closed and co-closed. Hence harmonic. $X$ is \emph{not} Killing. On the other hand, the distribution $X^\perp\subset TM$ is integrable.
\end{itemize}
 \end{prop}
\begin{proof}
We have $\na_XX=0$ and $\langle\na_\cdot X,X\rangle=0$. The matrix of $A$ restricted to $X^\perp$ is skew-symmetric and symmetric in the respective cases $\varphi_\pm$. In the positive case, there exists an orthonormal frame $Y,Z$ such that ${A=\scriptstyle{\left(\begin{array}{cc} 0 &\lambda \\ -\lambda &0 \end{array}\right)}}$. In the negative case, then $A$ has eigenfunctions $\lambda,-\lambda\neq0$. For both we may use the formula
 \[ \dx X^\flat= X^\flat\wedge\na_XX^\flat+   Y^\flat\wedge\na_YX^\flat+Z^\flat\wedge\na_ZX^\flat  .  \]
 To see if the orthogonal distribution to $X$ is integrable we are bound to check, for every $Y,Z$ orthonormal frame, that $\langle [Y,Z],X\rangle=-\langle\na_YX,Z\rangle+ \langle\na_ZX,Y\rangle=-(1\pm1)\langle\na_YX,Z\rangle$ vanishes. Which is obvious in one case.
 \end{proof}

 The following result generalizes to any curvature the well-known result from \cite{GluckZiller} for the 3-sphere. It has been proved by completely different means in \cite{GilMedrano2022,GilMedranoLlinaresFuster} and in \cite{DavilaVanhecke1,DavilaVanhecke2}.

\begin{teo} \label{Theorem_Hopfminimalforanycurvature}
A Hopf vector field on $\Omega\subset\Sphere^3(r)$ has minimal volume $(1+\frac{1}{r^2})\vol(\Omega)$ in its homology class. Such solution is unique in any given $\Omega$.

A global Hopf vector field has minimal volume $2\pi^2(r+r^3)$.
\end{teo}
\begin{proof}
 We start by uniqueness. Equations (i-ii) in Proposition  \ref{Propfundamental} yield $A_{0j}=0$, $A_{11}=A_{22}$, $A_{12}=-A_{21}$.
 Applying formula \eqref{normofvarphivectorfield} with our choice of $\varphi_+$, we find the 3-form $X^*\varphi_+$ is $(1+{A_{11}}^2+{A_{12}}^2)\vol_M$. Uniqueness comes from invariance (not from K\"ahler geometry as in the case $r=1$). Indeed, $A_{ij}$ must be constant in order to have minimality everywhere, ie. every $\Omega\subset\Sphere^3(r)$ with free boundary conditions.

 Those equations are satisfied by a Hopf vector field. Indeed we deduce in Section 4 that $A_{0j}=0$, $A_{11}=A_{22}=0$, $A_{12}=-A_{21}=\frac{1}{r}=\sqrt{c}$. These are the only possible values and every such vector field $X$ is a Hopf vector field.

 For the last assertion, we recall $\vol(\Sphere^3(r))=2\pi^2r^3$.
\end{proof}

Hopf vector fields are induced by linear complex structures $J_0$ on $\R^4$ compatible with the metric and orientation. On the 3-sphere, $X_x=J_0x$.

Recalling $\mathrm{Isom}_+(\Sphere^3)=\SO(4)$ the question for the quotient manifolds of the sphere arises. This has also deserved attention in \cite{BorZou} and again recently in \cite{GilMedrano2022}.
\begin{coro}
 Given a subgroup $\Gamma\subset\SO(4)$ acting freely on the 3-sphere, the manifold $\Sphere^3/\Gamma$ has as minimal vector fields the Hopf vector fields for which $[\Gamma,J_0]=0$.
\end{coro}
The proof follows from the same theory of calibrations which gave Theorem \ref{Theorem_Hopfminimalforanycurvature}, since the Levi-Civita connection on $M$, the Sasaki metric and the fundamental differential system on $T^1M$ are all invariant by isometries, for any oriented Riemannian manifold $M$.

\cite{BorZou} proved that every finite subgroup $\Gamma$ acting freely and by isometries on the sphere admits a $J_0$ such that $\Gamma<\mathrm{U}(\C^2,J_0)$, this is $[\Gamma,J_0]=0$.


Proposition \ref{Propfundamental} gives also a solution to the case of the flat metric (obviously the parallel vector fields). Indeed, notice we may use both $\varphi_-\sim\varphi_+\nsim0$ in this case, cf. Proposition \ref{cohomologous3formscneq1} (iii).

\begin{center}
\textsc{3.2 -- No solutions}
\end{center}

The above result of a solution of the minimal volume vector field on $\Sphere^3$ as the geodesic flow of the 2-sphere, as well as the appearance in the hyperbolic case of a system of differential equations (i-ii) leading to strong similarity with Anosov flows, which furthermore require a \textit{hyperbolic set for the flow}, cf. \cite[Definition 17.14.1]{KatokHasel}, and yield precisely that a geodesic flow on hyperbolic metric is Anosov, cf. \cite[p. 346]{Chavel}\cite{KatokHasel}, leads us to ask if, for the case $c=-1$, a local solution $X$ of minimal volume on $\Hyperbolic^3$ exists as the geodesic flow of some unknown surface yet to be discovered. However, no such analogy is true.

\begin{prop}  \label{prop_nonexistenceminimalvolumevectorfield}
On a given domain of a negative constant sectional curvature manifold there exists no vector field $X$ which is a solution of (i-ii) from Proposition \ref{Propfundamental}.
\end{prop}
\begin{Rema}
 Equations (i-ii) are invariant under rotation of $Y,Z$ in both cases $\pm$.
\end{Rema}
\begin{proof}[A wise proof just for case $\varphi_-$]
 Suppose a solution exists on a non-empty open subset. Notice $X$ is locally a gradient, since $X^\flat$ is closed, and moreover $X$ is harmonic. Then a well-known result of Cartan, owing to the search for isoparametric hypersurfaces, states that there do not exist harmonic functions with gradient of constant norm on hyperbolic space, cf. \cite{Plaue} and the references therein.
\end{proof}
\begin{proof}[Proof of both cases]
 Let us see $\varphi_-$ first, again.
 Suppose $X$ exists on an open subset $\neq\emptyset$. 
 Recall $\na_XX=0$ and $\langle\na_\cdot X,X\rangle=0$. The matrix of $A$ restricted to $X^\perp$ is symmetric and has eigenfunctions $\lambda,-\lambda\neq0$. For two local orthonormal vector fields $Y,Z$ in $X^\perp$, we have $\na_YX=\lambda Y$, $\na_ZX=-\lambda Z$ (this is the kind of phenomena observed with an Anosov flow, cf. \cite[Chapter 17.4]{KatokHasel}). Then it follows that $\na_YZ=\mu Y$ and $\na_ZY=\nu Z$ for some functions $\mu,\nu$, and therefore $\na_YY=-\mu Z-\lambda X$ and $\na_ZZ=\lambda X-\nu Y$. In the same way we find $\na_XY=pZ$, $\na_XZ=-pY$ for some function $p$. Regarding the Lie brackets, we get
 \[  [X,Y]=pZ-\lambda Y,\quad [X,Z]=-pY+\lambda Z,\quad [Y,Z]=\mu Y-\nu Z. \]
 Then $R(X,Y)X=\dx\lambda(X)Y+2\lambda pZ+\lambda^2 Y$ and, from constant $-1$ sectional curvature tensor, it follows that $p=0$ and $\dx\lambda(X)+\lambda^2=1$. From $R(X,Y)Z=\dx\mu(X)Y+\lambda\mu Y=0$, follows $\dx\mu(X)+\lambda\mu=0$. As with $Y$, we have $R(X,Z)X=-\dx\lambda(X)Z+\lambda^2Z$, and therefore we find now $-\dx\lambda(X)+\lambda^2=1$. We conclude $\lambda^2=1$. Continuing, from $R(X,Z)Y=\dx\nu(X)Z-\lambda\nu Z$, follows $\dx\nu(X)-\lambda\nu=0$. From $R(Y,Z)X=-2\mu\lambda Y-2\nu\lambda Z=0$, comes that $\mu=\nu=0$. Finally, a last deduction $R(Y,Z)Y=-\lambda^2Z$ implies that $\lambda^2=-1$.

 Now the case $\varphi_+$. We suppose $X$ exists such that $\na_XX=0$ and that for a local orthonormal frame $Y,Z$ of $X^\perp$ we have now $\na_YX=\lambda Z,\ \na_ZX=-\lambda Y$ for some function $\lambda$. We also see $\na_XZ=\mu Y$ and $\na_XY=-\mu Z$ for some function $\mu$. Hence $[X,Z]=\mu Y+\lambda Y$ and it follows that $R(X,Z)X=-\dx\lambda(X)Y-\lambda^2Z$ and therefore $\lambda^2=-1$, a same contradiction as above.
\end{proof}
It is not proved above that there do not exist minimal volume vector fields e.g. on a constant hyperbolic manifold of finite volume, but it seems reasonable to conjecture so. We may conjecture the infimum $(1-c)\vol(M)$, the lower bound for compact $M$ found in \cite[Corollary 4]{BritoChaconNaveira} (cf. also \cite[Proposition 13]{GilMedranoLlinaresFuster}).

If the circle ${b_0}^2+{b_1}^2=1$ of calibrations from Proposition \ref{circleofcalibrations} may still yield the right equations for the problem of minimal volume on hyperbolic ambient, then we should probably notice that always some $\varphi_t\sim\varphi_+$, for some $t$, and that $\varphi_+$ has no solutions.

Case $c=-1$ is specially twisted, since $\varphi_{\frac{\pi}{2}}=\theta\wedge\alpha_1=\dx\alpha_0\sim0\sim\varphi_+$. Although this case $b_0=0,\ \pm b_1=1$ leads to complicated conditions on $X$ in general, we shall give a simple solution in Section 5.


\begin{center}
\textsc{3.3 -- The original idea for $\Sphere^3(1)$}
\end{center}

Let us finally clarify the coincidence of our system of differential forms on $T^1\Sphere^3(1)=V_2(\R^4)$ with that one of Gluck and Ziller in the case of the 3-sphere of radius 1. Recall their finding of two closed 2-forms $\tau_1,\tau_2$, mutually orthogonal pullbacks of the K\"ahler forms of two $\Sphere^2$, and from which a calibration is defined:
\begin{equation}
    \varphi=\theta\wedge(\tau_1-\tau_2) .
\end{equation}
More explicitly $\tau_1=f^{12}$,\ $\tau_2=f^{34}$ arise from a coframe of horizontals $f^1,f^2$ and verticals $f^3,f^4$, askew to the usual Sasaki splitting of $e_0^\perp$. Now letting
\[   f^1=e^1+e^4,\quad f^2=e^2-e^3,\quad f^3=e^2+e^3,\quad f^4=e^1-e^4 ,  \]
it is easy to see
\[  \tau_1= f^{12} =  \alpha_0 + \alpha_2 + \dx\theta \qquad
-\tau_2=- f^{34} =  \alpha_0 + \alpha_2 - \dx\theta .   \]
By \ref{omega_c}, these are closed forms. Moreover,
$\varphi=\theta\wedge(\tau_1-\tau_2)=2\theta\wedge(\alpha_0+\alpha_2)=2\varphi_+$ from \eqref{veryclosedcalibration}.

Hence the calibrations are essentially the same.

\begin{center}
\subsection*{4 -- On the geodesic flow and Hopf vector fields}
\end{center}

We think of the non-Euclidean space forms as hypersurfaces of the Euclidean or Lorentzian space $\R^{m+1}$. Defining
\begin{equation}
 \langle\ ,\ \rangle_\pm=\pm(\dx x^1)^2+(\dx x^2)^2+\cdots+(\dx x^{m+1})^2
\end{equation}
we let $M_+=\Sphere^m(r)=\{x:\ \langle x,x\rangle_+=r^2\}$ and $M_-=\Hyperbolic^m(r)=\{x:\ \langle x,x\rangle_-=-r^2,\ x^1>0\}$, where $r>0$. These metrics restrict to positive definite metrics in both subspaces. They are also given as the invariant metrics. In particular, $M_-=\SO_0(1,m)/\SO(m)$.

Next we prove a result which is asserted in the case of spheres. The proof would seem quite immediate through the exponential map, but we wish to see it in some detail.

Let $x\in M_\pm$ and $(x,y)\in T^1_xM_{\pm}$. Hence $y$ is a unit vector, $\langle y,y\rangle_\pm=1$, orthogonal to $x$.

\begin{teo} \label{Theorem_velocity_proof}
The geodesic spray vector fields $e_0$ of $M_+$ or $M_-$ along the curves in $T^1M_+$ and $T^1M_-$, respectively,
\begin{equation}  \label{twogeodesicflowcurves}
\begin{split}
 &  g_t(x,y)=(x\cos t+yr\sin t,-x\frac{\sin t}{r}+y\cos t), \\
 &  g_t(x,y)=(x\cosh t+yr\sinh t,x\frac{\sinh t}{r}+y\cosh t) ,
\end{split}
\end{equation}
agree with $\frac{1}{r}$ times the velocity of these same curves.
\end{teo}
\begin{proof}
We have $\pi:TM_\pm\lrr M_\pm$, $\pi(x,y)=x$, where $\langle x,x\rangle_\pm=\pm r^2$. Differentiating this equation, it follows that $x\perp_\pm y$. Next we let drop the $\pm$ where there is no fear of confusion.

In the same way, it is easy to see
\begin{equation*}
 T(TM)=\bigl\{(x,y,u,v)\in(\R^{m+1})^4:\ \langle x,x\rangle=\pm r^2,\ \langle x,y\rangle=\langle u,x\rangle=\langle u,y\rangle+\langle v,x\rangle=0\bigr\}.
\end{equation*}
Vertical vectors are those for which $u=0$. The tautological vector field $\xi$ is clearly given by $\xi_{(x,y)}=(x,y,0,y)$. The mirror map $B\in\End{TTM}$ is simply the map $B(x,y,u,v)=(x,y,0,u)$.

The Levi-Civita connections of the given hypersurfaces are found to be, respectively, $\na_XY=\dx Y(X)\pm\langle X,Y\rangle \frac{x}{r^2}$ at a point $x\in M$. Indeed well-defined, it is the metric torsion-free connection. Invariant under the respective $\SO(m+1)$ and $\SO_0(1,m)$.

Notice we have started to suppress the base point in $TM$; this will be even more considerable when referring to $TTM$.

A simple computation yields $R^\na(X,Y)Z=\pm\frac{1}{r^2}(\langle Y,Z\rangle X-\langle  X,Z\rangle Y)$, and hence the conclusion that $M_\pm$ has constant sectional curvature $c=\pm\frac{1}{r^2}$.

Regarding the tangents just to $T^1M_\pm$ we have, moreover, $\langle y,y\rangle_\pm=\|y\|^2=1$, $\langle y,v\rangle=0$. We then lift the Levi-Civita connection of $M$ to a metric connection $\na^*$ over $(TM,g^S)$, by taking the pullback connection to $\pi^\star TM=\ker\dx\pi$, the vertical subspace. For any vector field $\fraku$ over $TM$, we have
\begin{equation*}
 \fraku_{(x,y)}=(x,y,U_1,U_2) \quad\mbox{and}\quad \dx\pi(\fraku_{(x,y)})=(x,U_1)
\end{equation*}
and thus
\begin{equation*}
 \na^*_\fraku\xi=(x,y,0,U_2\pm\langle U_1,y\rangle \frac{x}{r^2}).
\end{equation*}
One verifies that $\langle\xi,\na^*_\fraku\xi\rangle=\langle y,U_2\rangle=0$ over $T^1M$, as expected. The formula also gives the horizontal subspace as the kernel:
\begin{equation}\label{eqhorizontal}
 \fraku\in H^\na\quad\mbox{if and only if}\quad r^2U_2\pm\langle U_1,y\rangle x=0.
\end{equation}
It is  easy to find the horizontal and vertical lifts of $X=(x,U)$ tangent to $M$ at $x$. For instance, the horizontal lift of $X$ is $\pi^*X=(x,y,U,\mp\langle U,y\rangle\frac{x}{r^2})$. As expected, $B\pi^*X=\pi^\estrela X$.

Finally the geodesic spray $e_0=B^\mathrm{t}\xi$ equals $e_0=(x,y,y,\mp\frac{x}{r^2})$.
Notice the defining conditions are satisfied: $e_0$ is horizontal and $\dx\pi({e_0}_{X_x})=X_x$, for any $X$ tangent to $M$ at $x$.

Recalling \eqref{geodesic_flow_on_Sm}, which exists in any dimension,
we may easily deduce, that at any point $g_t(x,y)=(x\cos t+yr\sin t,-x\frac{\sin t}{r}+y\cos t)$ lies in $T^1\Sphere^m(r)$. And secondly, that
\begin{eqnarray*}
 \papa{ }{t}g_t(x,y)&=
 &(-x\sin t+yr\cos t,-x\frac{\cos t}{r}-y\sin t) \\
 &=&  r\bigl(-x\frac{\sin t}{r}+y\cos t, -\frac{1}{r^2}(x\cos t+yr\sin t)\bigr)
 \\ &=& r e_0\,_{|_{g_t(x,y)}}
 \end{eqnarray*}

The hyperbolic case is just as easy: $g_t(x,y)$ lies in $T^1\Hyperbolic^m(r)$ and $\papa{ }{t}g_t(x,y)= r e_0\,_{|_{g_t(x,y)}}$.
\end{proof}

We have again the well-defined and non-singular projections
\begin{equation}
   V_{2,\pm}(\R^{m+1})=T^1M \lrr \mathrm{Gr}_2(\R^{m+1}), \quad (x,y)\longmapsto x\wedge y,
\end{equation}
with kernel the curves in \eqref{twogeodesicflowcurves}.

We have the $\SO(2)$ action on $T^1M_+$ arising from the geodesic flow. This action is by isometries, as pointed in \cite{GluckZiller}, for the case $r=1$. Moreover,  it is asserted in \cite{Chavel} that the action of the general geodesic flow is by isometries if and only if we are in the case of sectional curvature 1. We shall confirm this quite explicitly.

Notice that the geodesic flow of hyperbolic metric is not always complete, since it must lie in $\{x^1>0\}$.

Following the definition in Theorem \ref{Theorem_velocity_proof} we still observe the (pseudo) Lie group action of $\R$ on $T^1M_-$. Indeed the identities $g_0=1,\ g_{t_1+t_2}=g_{t_1}\circ g_{t_2}$ hold true for any $t_1,t_2\in\R$.

\begin{teo}   \label{Theorem_wideraction}
 The Lie groups $\SO(m+1)$ and $\SO_0(m,1)$ act by isometries on $T^1\Sphere^m(r)$ and $T^1\Hyperbolic^m(r)$, respectively, leaving the fundamental differential system from \cite{Alb2019b} invariant.

 The actions of $\SO(2)$ and $\R$ as geodesic flows commute with the previous respective symmetric space structures. These actions are by isometries if and only if we are in the case of $\Sphere^m(1)$.
\end{teo}
\begin{proof}
 Any isometry of any given Riemannian manifold leaves the fundamental differential system invariant. This stems from the Levi-Civita connection and therefore the Sasaki metric being invariant by isometry.

  We are left with the proof of the second assertion. Since $g_t$ acts linearly, at a point $(x,y)$ we have $\dx g_t(u,v)=g_t(u,v) \in T^1_{p_{1}(g_t(x,y))}M$. With $X=(x,U),\ x\perp U$, we have
 \[ \dx g_t(\pi^*X)=g_t(U,\mp\langle U,y\rangle\frac{x}{r^2})=
  \bigl(Uc\mp\langle U,y\rangle\frac{x}{r}s,\mp U\frac{s}{r}\mp\langle U,y\rangle   \frac{x}{r^2}c\bigr) , \]
 where $c=\cos t,\ s=\sin t$ in the elliptic case and $c=\cosh t,\ s=\sinh t$ in the hyperbolic case. At point $g_t(x,y)=(xc+yrs,\mp x\frac{s}{r}+yc)$, the vertical part is, cf. \eqref{eqhorizontal}:
 \begin{eqnarray*}
  \lefteqn{ \mp U\frac{s}{r}\mp\langle U,y\rangle
  \frac{x}{r^2}c\pm\frac{1}{r^2}\Bigl\langle Uc\mp\langle U,y\rangle\frac{x}{r}s,\mp x\frac{s}{r}+yc\Bigr\rangle(xc+yrs) \ \ }\\
  &=& \mp U\frac{s}{r}\mp\langle U,y\rangle\frac{c}{r^2}x\pm\frac{c^2}{r^2} \langle U,y\rangle(xc+yrs)+\frac{s^2}{r^2}\langle U,y\rangle(xc+yrs) \\
  &=& \mp\frac{s}{r}(U-\langle U,y\rangle y).
 \end{eqnarray*}
Now for verticals $(x,y,0,U_2),\ U_2\perp x,y$, we have $\dx g_t(0,U_2)=(U_2rs,U_2c)$, which has vertical part
\[  U_2c\pm\frac{1}{r^2}\langle U_2rs,\mp x\frac{s}{r}+yc\rangle(xc+yrs)=U_2c  \]

Finally, with any $U,U_1\perp x$ giving horizontal vector fields and any $U_2,U_3$ giving verticals:
 \begin{equation*}
  \langle \dx g_t(U,\mp\langle U,y\rangle\frac{x}{r^2}),\dx g_t(U_1,\mp\langle U_1,y\rangle\frac{x}{r^2}) \rangle = (c^2+\frac{s^2}{r^2})\langle U,U_1\rangle+ (\pm s^2-\frac{s^2}{r^2})\langle U,y\rangle\langle U_1,y\rangle,
 \end{equation*}
\[   \langle\dx g_t(0,U_2),\dx g_t(0,U_3)\rangle = (r^2s^2+c^2)\langle U_2,U_3\rangle \]
and
 \begin{equation*}
  \langle\dx g_t(U,\mp\langle U,y\rangle\frac{x}{r^2}),\dx g_t(0,U_2)\rangle = (crs\mp\frac{cs}{r})\langle U,U_2\rangle .
 \end{equation*}
 $g_t$ is an isometry for the Sasaki metric if and only if the three formulas above are, respectively, equal to $\langle U,U_1\rangle,\ \langle U_2,U_3\rangle$ and $0$. The result follows immediately.
\end{proof}

When $m=2k+1$ is odd we have further understanding of the geodesic flow.

We see it by means of the Hopf vector fields. Let $J_0\in\End{\R^{2k+2}}$ be a linear complex structure compatible with the metric. It gives rise to a Hopf vector field, defined, at each $x\in\Sphere^m(r)$, by $X=\frac{1}{r}J_0x$. We have omitted the base point on the target manifold. Clearly
\begin{equation}\label{identitiesforHopfvf}
 \na_{W}X=\dx X(W)+\langle X,W\rangle \frac{x}{r^2}=\frac{1}{r} J_0W+\langle X,W\rangle \frac{x}{r^2}
\end{equation}
and so $\na_XX=0$. Finally, for any $Y$ such that $Y,J_0Y\perp \{x,X\}$, we have successively
\begin{equation}\label{identitiesforHopfvf2}
 \na_YX=\frac{1}{r} J_0Y,\qquad \langle\na_YX,Y\rangle=0,\qquad\langle\na_YX,J_0Y\rangle=\frac{1}{r}\|Y\|^2=-\langle\na_{J_0Y}X,Y\rangle.
\end{equation}
Reciprocally, a vector field $X$ on the sphere which satisfies \eqref{identitiesforHopfvf2} must be a Hopf vector field.

These equations are thus very useful in regard to Proposition \ref{Propfundamental}.

In the case of hyperbolic space, we \textit{could} follow exactly the same steps taking $J_0$ compatible with the metric. Again $X=J_0x\perp x\in\Hyperbolic^{2k+1}$. And $\na_{W}X=J_0W-\langle X,W\rangle x$. And all the identities in \eqref{identitiesforHopfvf2} would be accordingly satisfied. However, such a $J_0$ does not exist in $\SO(1,2k+1)$.

\begin{center}
\subsection*{5 -- Towards minimal vector fields on hyperbolic metric}
\end{center}

Let us now assume that $M$ is any given oriented Riemannian 3-manifold, not necessarily of constant sectional curvature. Regarding the $\SO(2)$ decomposition of the 10-dimensional representation space $\Lambda^3\R^5\simeq\Lambda^2\R^5$, the 1st-order structure equations satisfy the following, cf. \cite{Alb2018a}:
\begin{equation}\label{equacoesestruturaisgerais}
 \dx\alpha_0=\theta\wedge\alpha_1,\qquad\dx\alpha_1=2\theta\wedge\alpha_2-r\,\theta\wedge\alpha_0,\qquad\dx\alpha_2=\theta\wedge\gamma-\frac{r}{2}\,\theta\wedge\alpha_1+\alpha_0\wedge\rho,
\end{equation}
where $\rho$ is the vertical pullback
\begin{equation}
 \rho=\xi\lrcorner\pi^{\star}\ric=R_{1012}e^4-R_{2012}e^3
\end{equation}
and $r$ is the function on $T^1M$ defined by
\begin{equation}
 r(u)=\ric(u,u),\ \ \ u\in T^1M.
\end{equation}
Regarding $\gamma$, we refer to \cite{Alb2018a,Alb2019b} for the definition and other results. We recall the above identities have all been proved in two different ways.

Proposition \ref{Propfundamental} together with \eqref{equacoesestruturaisgerais} yields a remarkable result.
\begin{teo}
 For every oriented Riemannian 3-manifold $M$, the total space $T^1M$ carries two natural degree 3 calibrations (up to a $\pm$ sign):
\begin{equation}
 \theta\wedge\alpha_0\qquad\mbox{and}\qquad\dx\alpha_0=\theta\wedge\alpha_1 .
\end{equation}
\end{teo}
Notice $\pm\theta\wedge\alpha_1$ is of the kind found in Proposition \ref{circleofcalibrations}.

From Lemma \ref{Lemma_varphicalibratedvectorfield}, a vector field is a calibrated submanifold of $\theta\wedge\alpha_0$ if and only if it is parallel.

More important now is the solution of the optimal vector field problem on hyperbolic space.

 Consider the following model of hyperbolic metric on $M=\R^2\times\R_+$, with $a>0$ fixed, and constant sectional curvature $c=-a$:
 \[  \langle\ ,\ \rangle=\frac{1}{at^2}\bigl((\dx x^1)^2+(\dx x^2)^2+(\dx t)^2\bigr) .\]

\begin{coro}
For $c=-1$, the vector field $X\in\XIS^1_\Omega$ defined on a domain $\Omega\subset M$ by $X=t\partial_t$, has minimal volume
\[  \vol(X)=2\vol(\Omega)=-\int_{\partial\Omega}t^4\,\dx x^1\wedge\dx x^2  \]
among all vector fields in its homology class.
\end{coro}
\begin{proof}
 It is easy to see the Levi-Civita connection is given by
 \begin{equation*}
  \na_{\partial_i}\partial_j=\frac{1}{t}\delta_{ij}\partial_t ,\quad\quad
 \na_{\partial_i}\partial_t=\na_{\partial_t}\partial_i=-\frac{1}{t}\partial_i ,\quad\quad \na_{\partial_t}\partial_t=-\frac{1}{t}\partial_t , \qquad\forall i,j=1,2,
 \end{equation*}
 and to check that the metric has constant sectional curvature $-a$.

 Now we use the exact calibration $\varphi=-\theta\wedge\alpha_1$ and recall Lemma \ref{Lemma_varphicalibratedvectorfield} with $X=X_0=\sqrt{a}t\partial_t$. For this, we let $X_1=\sqrt{a}t\partial_1$, $X_2=\sqrt{a}t\partial_2$, composing the orthonormal frame. We find $\na_{\partial_t}t\partial_t=\partial_t-t\frac{1}{t}\partial_t=0$ and $\na_{\partial_i}t\partial_t=-t\frac{1}{t}\partial_i=-\partial_i$ ($i=1,2$). Hence, for the $A_{ij}=\langle\na_{X_i}X,X_j\rangle,\ i,j=0,1,2$, we obtain
 \[  A_{i0}=0,\qquad A_{0j}=\langle\na_XX,X_i\rangle=0,\qquad A_{ij}=\langle\na_{X_i}X,X_j\rangle=0,\ \ i\neq j ,\]
\[  A_{ii}=\langle\na_{X_i}X,X_i\rangle= a^{\frac{3}{2}}t^2\langle\na_{\partial_i}t\partial_t,\partial_i\rangle=-\sqrt{a},\ \ i=1,2. \]
Finally, equation \eqref{varphicalibratedvectorfield} reads
\[  -(-2\sqrt{a})=\sqrt{1+a+a+a^2} ,  \]
which has the desired solution with $a=1$. Moreover, in this case,
 \[ \vol(X)=\int_\Omega X^*\varphi=-\int_{\partial\Omega}X^*\alpha_0=
  -\int_{\partial\Omega}X\lrcorner\vol_M=-\int_{\partial\Omega}t^4\,\dx x^1\wedge\dx x^2,
 \]
since the volume form of $M$ is $t^3\,\dx x^1\dx x^2 \dx t$.
\end{proof}

Recalling the Introduction, we seem to have a one-time-deal too with $c=-a=-1$ as in \cite{GluckZiller} with $c=1$. The volume of $X$ as twice the volume of $\Omega$ was pointed in Section 3.2 for \textit{compact} $M$.

The last result is a bit curious since $X_1$ or $X_2$ have less volume than $X_0$ in any domain. In order to see this, let us reshape the indices:
\[   X=t\partial_1, \qquad Y=t\partial_2,\qquad Z=t\partial_t  .\]
Then $ \na_XX=t^2\na_{\partial_1}\partial_1=t\partial_t=Z,\
  \na_YX=t^2\na_{\partial_2}\partial_1=0,\
 \na_ZX=t\na_{\partial_t}t\partial_1=t\partial_1+t^2\na_{\partial_t}\partial_1= 0$.
And therefore, with the usual ordering,
$A_{00}=A_{01}=A_{1j}=A_{2j}=0,\ \forall j$, and only $A_{02}=1$. Whence $\vol(X)=\sqrt{2}\vol(\Omega)$. We note that no calibration studied in this article can deal with this vector field.

The conclusion therefore must be that the minimal vector field $t\partial_t$ and $t\partial_1$ have different homology representatives when restricted to the boundary.

We have not found in the literature any other explicit solution for minimal vector fields on hyperbolic metric. It should be an interesting problem concerning 3-manifold geometry.


\vspace{7mm}

\begin{center}
 Statements and Declarations
\end{center}

\noindent Competing interests: The authors have no competing interests to declare that are relevant to the content of this article.

\vspace*{1mm}

\noindent Funding: The research leading to these results has received funding from Funda\c c\~ao para a Ci\^encia e a Tecnologia from the Portuguese Republic. Project Ref. UIDB/04674/2020. \url{https://doi.org/10.54499/UIDB/04674/2020}


\vspace*{13mm}

\textsc{R. Albuquerque}\ \ \ \textbar\ \ \ 
{\texttt{rpa@uevora.pt}}

Centro de Investiga\c c\~ao em Mate\-m\'a\-ti\-ca e Aplica\c c\~oes

Rua Rom\~ao Ramalho, 59, 671-7000 \'Evora, Portugal


\end{document}